\documentclass[amssymb, 11pt]{amsart}

\usepackage{latexsym}

\newlength{\standardunitlength}
\newtheorem{prop}{Proposition}[section]

\newtheorem{lemma}[prop]{Lemma}

\newtheorem{theorem}[prop]{Theorem}
\newtheorem{example}[prop]{Example}

\begin{document}

\title [Combinatorics of balanced carries] {Combinatorics of balanced carries}

\author{Persi Diaconis}
\address{Department of Mathematics and Statistics\\ Stanford University\\
        Stanford, CA, 94305}

\author{Jason Fulman}
\address{Department of Mathematics\\
        University of Southern California\\
        Los Angeles, CA, 90089}
\email{fulman@usc.edu}


\date{September 19, 2013}

\begin{abstract} We study the combinatorics of addition using balanced digits, deriving an analog of Holte's ``amazing matrix'' for carries
in usual addition. The eigenvalues of this matrix for base $b$ balanced addition of $n$ numbers are found to be $1,1/b,\cdots,1/b^{n}$,
and formulas are given for its left and right eigenvectors. It is shown that the left eigenvectors can be identified with hyperoctahedral Foulkes
characters, and that the right eigenvectors can be identified with hyperoctahedral Eulerian idempotents. We also examine the carries that occur when a column of balanced
digits is added, showing this process to be determinantal. The transfer matrix method and a serendipitous diagonalization are used to study this
determinantal process.
\end{abstract}

\maketitle

\section{Introduction} \label{intro}

This paper studies the combinatorics of ``carries'' in basic arithmetic, using balanced digits. To begin we describe the motivation for using balanced digits. When ordinary integers are added, carries occur. Consider a carries table with rows and columns indexed by digits $0,1,\cdots,b-1$ (working base $b$) and a carry at $(i,j)$ if $i+j \geq b$. Thus when $b=5$, labeling the rows and columns in the order $0,1,2,3,4$ the carries matrix is

\begin{equation*}
\begin{pmatrix}
0 & 0 & 0 & 0 & 0\\
0 & 0 & 0 & 0 & b \\
0 & 0 & 0 & b & b \\
0 & 0 & b & b & b \\
0 & b & b & b & b
\end{pmatrix}.
\end{equation*}

In general there are ${b \choose 2}$ carries (so 10 when $b=5$). If digits are chosen uniformly at random, the chance of a carry is ${b \choose 2}/b^2 = \frac{1}{2} - \frac{1}{2b}$.

The digits $0,1,\cdots,b-1$ can be thought of as coset representatives for $b \mathbb{Z} \subseteq \mathbb{Z}$ and the carries are cocycles \cite{I}. For $b$ odd (as assumed throughout this paper), consider instead the balanced representatives $0, \pm 1, \cdots \pm (b-1)/2$. One motivation for using balanced representatives is that they lead to fewer carries. For example, when $b=5$, writing $\overline{j}$ for $-j$, the digits are $\{0,1,\overline{1},2,\overline{2} \}$. Labeling the rows and columns in the order $\overline{2},\overline{1},0,1,2$, the carries matrix is

\begin{equation*}
\begin{pmatrix}
\overline{b} & \overline{b} & 0 & 0 & 0\\
\overline{b} & 0 & 0 & 0 & 0 \\
0 & 0 & 0 & 0 & 0 \\
0 & 0 & 0 & 0 & b \\
0 & 0 & 0 & b & b
\end{pmatrix}.
\end{equation*} (For example, $(-2) + (-2) = -4 = -5+1$). Here there are 6 carries versus 10 for the classical choice. For general $b$, balanced carries lead to
$(b^2-1)/4$ carries. This is the smallest number possible \cite{A},\cite{DSS}.

Balanced digits are elementary but unfamiliar: for example in base $5$, $13$ is equal to $1 \overline{2} \overline{2}$, and $-9$ is equal to $\overline{2} 1$.
Negating numbers negates the digits and the sign of the number is the sign of its left-most digit. Balanced digits were introduced in 1726 by Colson \cite{Co}; see Cajori \cite{Ca} or Chapter 4 of Knuth \cite{Kn} for history and applications.

Of course, balanced digits may be used for addition with larger numbers. For example:

\begin{small}
\begin{equation*}
\begin{array}{l}
\overline{1} 0 1 1 0 0 0\\
\\
1 \overline{2} 1 2 2 \overline{1} 0\\
1 \overline{1} \overline{1} 2 2 2 1 \\ \hline
\\
1 2 1 0 \overline{1} 1 1
\end{array}
\end{equation*}\end{small} Here the numbers along the top row are the carries. When two numbers are added the possible carries are
$0,1,\overline{1}$. If $n$ numbers are added, the possible carries are $- \lfloor \frac{n}{2} \rfloor, \cdots, \lfloor \frac{n}{2} \rfloor$.

Suppose now that balanced digits base $b$ are used, $n$ numbers are added, and that the digits are chosen uniformly at random in
$\{ \overline{\frac{b-1}{2}}, \cdots \frac{b-1}{2} \}$. Consider the carries along the top, from right to left as $\kappa_1,\kappa_2,\cdots$.
(Thus $\kappa_1=0$ always). It is easy to see that the $\{\kappa_i\}$ form a Markov chain on the set
$- \lfloor \frac{n}{2} \rfloor, \cdots, \lfloor \frac{n}{2} \rfloor$. For the classical choice of digits, this Markov chain
was analyzed by Holte \cite{Ho}, with follow-up reviewed later in this introduction.

Let $K(i,j)$ denote the transition matrix of the balanced carries Markov chain. Of course $K(i,j)$ depends on $b$ and $n$ but this is suppressed.
In the case that $n$ is odd, we find that the matrix $K(i,j)$ is the same as Holte's amazing matrix for usual carries. However when $n$ is even, new results emerge.
In particular, the stationary distribution of the Markov chain is given by
\[ \pi(j-n/2) =  \frac{\overline{A(n,j)}}{2^n n!} \ \ \ 0 \leq j \leq n.\] Here the $\overline{A(n,j)}$ are ``signed Eulerian numbers'',
defined carefully in Section \ref{lefteig}. In particular, the chance of a carry $-1,0,1$ is $\frac{1}{8}, \frac{3}{4}, \frac{1}{8}$ respectively.
We find explicit formulae for the left and right eigenvectors of the balanced carries chain; we show that the left eigenvectors can be identified with Miller's
hyperoctahedral Foulkes characters \cite{Mi}, and that the right eigenvectors can be identified with hyperoctahedral Eulerian idempotents of
Bergeron and Bergeron \cite{BB}. Before we finished writing this paper, the preprint \cite{NS} appeared, and there is some overlap with our work;
one can deduce our formula for the left eigenvectors of $K(i,j)$ from their paper.

Next we give a brief historical overview of the ``type A'' work that motivated this paper. The Markov chain of carries when $n$ random integers
are added mod $b$ (with the usual choice of digits) was first studied by Holte \cite{Ho}, who dubbed the matrix ``amazing''. He found that the
eigenvalues are $1,1/b,1/b^2,\cdots,1/b^{n-1}$, and identified the stationary distribution as $A(n,k)/n!$, with $A(n,k)$ the kth Eulerian number--the
number of permutations in $S_n$ with $k$ descents. This says what percent of carries are $k$ ($0 \leq k \leq n-1$). Holte further found closed
formulae for the left and right eigenvectors of the transition matrix.

The connection between carries and shuffling was developed by the present authors. The first proofs used generating functions and symmetric function theory
\cite{DF}, and an ``aha'' bijective proof was later found in \cite{DF2}. An algebraic combinatorics proof appeared in \cite{NT} and Pang \cite{P} studied the entire
descent pattern after shuffles using Hopf algebras. We note that Holte's amazing matrix also appears in algebraic geometry, giving the Hilbert series of the Veronese embedding \cite{BW}, \cite{DF}.

The connection of carries with Foulkes characters and Eulerian idempotents appears in \cite{DF3}, which identifies the left eigenvectors of Holte's matrix with the Foulkes characters of the symmetric groups, and the right eigenvectors with the Eulerian idempotents. The transition to other types of reflection groups is studied in Miller \cite{Mi}, and we identify the left eigenvectors of the balanced carries chain with Miller's hyperoctahedral Foulkes characters. We give a new proof of Miller's recurrence for
hyperoctahedral Foulkes characters. We also identify the right eigenvectors of the balanced carries chain with the Eulerian idempotents of the hyperoctahedral groups.
For a different proof connecting the inverse of the Foulkes character table with Eulerian idempotents, one can see \cite{Mi}.

The above results describe ``carries across the top'', when several long numbers are added. It is also fruitful to study the ``carries down a column'' when a single column of random digits is added. For ordinary addition, this was studied in \cite{borodin}, which showed that the positions of the carries form a determinantal point process, with explicitly computable correlation functions. We use the transfer matrix method and a serendipitous diagonalization to show that the same is true for ``carries down a column'' when a column of random balanced digits is added.

Carries and cocycles make sense for any subgroup $H$ of any group $G$. Choosing coset representatives $X$ for $H$ in $G$ and then picking elements $x$ in $X$ from some natural probability distribution leads to a carries process. This is developed in \cite{borodin} and \cite{DSS}. Developing a parallel theory involving a nested sequence of subgroups (as in the present paper) suggests a world of math to be done.

The organization of this paper is as follows.  Section \ref{amazing} begins by deriving an analog of Holte's amazing matrix for balanced carries. When an odd number of numbers is added, we show that this reduces to Holte's amazing matrix for ordinary carries, and when an even number of numbers is added, we show that it reduces to the type $B$ carries chain of \cite{DF}. The argument is similar to Holte's, and the result can also be deduced from the paper \cite{NS}. Section \ref{lefteig} shows that the eigenvalues of the carries chain are $1,1/b,\cdots,1/b^n$ and studies its left eigenvectors, giving an explicit formula and identifying them with Miller's hyperoctahedral Foulkes characters. Section \ref{righteig} gives a formula for its right eigenvectors, relating them to hyperoctahedral Eulerian idempotents. Section \ref{point} studies ``balanced carries down a column'' as a determinantal point process.

\section{Amazing matrix for balanced carries} \label{amazing}

We work in an odd base $b$, with digits $0, \pm 1, \pm 2, \cdots , \pm (b-1)/2$. For reasons which become clear in the remarks
following Theorem \ref{trans}, we add an even number $n$ of numbers. Then the
carries range from $-\frac{n}{2}$ to $\frac{n}{2}$, and form a Markov chain on the set $\{ - \frac{n}{2}, - \frac{n}{2} + 1, \cdots,
\frac{n}{2} -1, \frac{n}{2} \}$. Theorem \ref{trans} works out the transition matrix for this Markov chain.

\begin{theorem} \label{trans} Let $K(i,j)$ be the transition probability of the balanced carries Markov chain on the set
$\{ - \frac{n}{2}, - \frac{n}{2} + 1, \cdots, \frac{n}{2} -1, \frac{n}{2} \}$, corresponding to the addition of an even number $n$ of numbers.
Then $K(i,j)$ is equal to both:
\begin{enumerate}
\item The coefficient of $x^{jb+(n+1)(b-1)/2-i}$ in $(1+x+\cdots+x^{b-1})^{n+1}/b^n$.

\item \[ \frac{1}{b^n} \sum_{l=0}^{\lfloor j+ \frac{n+1}{b} \frac{b-1}{2} - \frac{i}{b} \rfloor}
(-1)^l {n+1 \choose l} {n+jb+(n+1) \frac{b-1}{2} - i - lb \choose n} .\]
\end{enumerate}
\end{theorem}

As an example of Theorem \ref{trans}, when $n=2$, with the rows indexed by $i=-1,0,1$ and the columns indexed by $j=-1,0,1$, the transition matrix is, for all odd $b$,
\begin{equation*}
K(i,j)=\frac{1}{b^2}\begin{pmatrix}
\frac{b^2+4b+3}{8} & \frac{3}{4} (b^2-1) & \frac{b^2-4b+3}{8}\\
\frac{b^2-1}{8} & \frac{3b^2+1}{4} & \frac{b^2-1}{8}\\
\frac{b^2-4b+3}{8} & \frac{3}{4} (b^2-1) & \frac{b^2+4b+3}{8}
\end{pmatrix}.
\end{equation*}

\begin{proof} Suppose that the carry into a column is $i$, with $-\frac{n}{2} \leq i \leq \frac{n}{2}$. Let the $n$ digits in the column
be $X_1,\cdots,X_n$, with $-(b-1)/2 \leq X_i \leq (b-1)/2$ for all $i$. Then the carry to the next column is $j$ precisely if
\[ jb- (b-1)/2 \leq i+X_1+\cdots+X_n \leq jb+ (b-1)/2.\] Letting $X_i' = X_i + (b-1)/2$ for all $i$, one has that
$0 \leq X_1', \cdots, X_n' \leq b-1$, and that the carry to the next column is $j$ exactly when
\[ jb - (b-1)/2 \leq i + X_1' + \cdots + X_n' - n(b-1)/2 \leq jb + (b-1)/2,\] which is equivalent to
\[ jb + (n-1)(b-1)/2 - i \leq X_1' + \cdots + X_n' \leq jb + (n+1)(b-1)/2 - i .\]

Thus $K(i,j)$ is equal to $1/b^n$ multiplied by the number of solutions to
\[ X_1'+\cdots+X_n' + Y = jb + (n+1)(b-1)/2 - i, \] where $0 \leq Y \leq b-1$. This is equal to $1/b^n$
multiplied by the coefficient of $x^{jb+(n+1)(b-1)/2-i}$ in $(1+x+\cdots+x^{b-1})^{n+1}$, proving part 1.

For part 2, let $[x^a] f(x)$ denote the coefficient of $x^a$ in $f(x)$. Then by part 1,
\begin{eqnarray*}
K(i,j) & = & \frac{1}{b^n} [x^{jb+(n+1)(b-1)/2-i}] \left( \frac{1-x^b}{1-x} \right)^{n+1} \\
& = & \frac{1}{b^n} \sum_{l \geq 0} (-1)^l {n+1 \choose l} [x^{jb+(n+1)(b-1)/2-i-lb}] (1-x)^{-(n+1)} \\
& = & \frac{1}{b^n} \sum_{l=0}^{\lfloor j+ \frac{n+1}{b} \frac{b-1}{2} - \frac{i}{b} \rfloor}
(-1)^l {n+1 \choose l} {n+jb+(n+1) \frac{b-1}{2} - i - lb \choose n}.
\end{eqnarray*}
\end{proof}

{\it Remarks:}
\begin{enumerate}
\item The type B carries chain in the paper \cite{DF} has the same transition matrix as the balanced chain, when the number of numbers being added is even. More precisely the chain in the paper \cite{DF} has state space $\{0,1,\cdots,n\}$, and the balanced chain has state space $\{-n/2,\cdots,0,\cdots,n/2\}$. The chance that the type B chain in \cite{DF} (with $b$ replaced by $(b-1)/2$) goes from $i$ to $j$ is the same as the chance that the balanced chain goes from $i-n/2$ to $j-n/2$. This follows immediately by comparing the formula in part 1 of Theorem 4.2 of \cite{DF} with the formula in Theorem \ref{trans}.

\item Consider the balanced carries chain when adding an odd number $n$ of numbers. This is a Markov chain on the set
$\{-(n-1)/2, \cdots, (n-1)/2\}$, with transition probabilities given by Theorem \ref{trans} above. This is the same as Holte's
carries chain \cite{Ho} on $\{0,1,\cdots,n-1\}$. More precisely, for all $0 \leq i,j \leq n-1$, the chance that Holte's chain
moves from $i$ to $j$ is equal to the chance  that the balanced chain moves from $i-(n-1)/2$ to $j-(n-1)/2$. This follows by
comparing the formula in Holte's paper with the formula in Theorem \ref{trans} above.

\item Letting $K_b$ denote the base $b$ balanced carries transition matrix, one has that $K_a K_b = K_{ab}$. This follows from the fact (proved in Sections \ref{lefteig} and \ref{righteig}) that the eigenvalues of $K_b$ are $1,1/b,\cdots,1/b^n$, and that the eigenvectors are independent of $b$.

\item While it is not emphasized here, the papers \cite{DF}, \cite{DF2} develop a card shuffling interpretation of the transition matrix and a host of
card shuffling interpretations of the spectral properties developed here.

\item Balanced arithmetic can be developed for even bases. For example, when $b=10$ choose digits $0,\pm 1, \pm 2, \pm3, \pm 4, 5$ (or replace $5$ by $-5$). The results
seem similar but we have not fully worked out the details.

\end{enumerate}

\section{Left eigenvectors of the amazing matrix for balanced carries} \label{lefteig}

This section studies the left eigenvectors of the balanced carries matrix when an even number $n$ of numbers are added.
The eigenvalues turn out to be $1,1/b,1/b^2, \cdots, 1/b^n$, and we derive an explicit formula for the left eigenvectors,
identifying them with hyperoctahedral Foulkes characters. In particular, these eigenvectors turn out to be independent of $b$.
The left and right eigenvectors have myriad uses for quantifying rates of convergence and the behavior of features of the
carries process. These are detailed in Section 2 of \cite{DPR}.

\begin{theorem} \label{left} Let $K$ be the transition matrix of the balanced carries chain of Section \ref{amazing}. The $jth$ left
eigenvector of $K$ (where $0 \leq j \leq n$), corresponding to the eigenvalue $1/b^j$, evaluated at the
state $i$ (where $-n/2 \leq i \leq n/2$), is given by
\[ v_j^n[i] = \sum_{r=0}^{i+n/2} (-1)^r {n+1 \choose r} (n+2i-2r+1)^{n-j} .\]
\end{theorem}

For example, when $n=2$, the matrix whose rows are the left eigenvectors of $K$ with eigenvalue $1/b^j$ (with $0 \leq j \leq 2$)
is given by

\begin{equation*}
\begin{pmatrix}
1 & 6 & 1 \\
1 & 0 & -1 \\
1 & -2 & 1
\end{pmatrix}.
\end{equation*}

When $n=4$, the matrix of left eigenvectors is

\begin{equation*}
\begin{pmatrix}
1 & 76 & 230 & 76 & 1 \\
1 & 22 & 0 & -22 & -1 \\
1 & 4 & -10 & 4 & 1 \\
1 & -2 & 0 & 2 & -1 \\
1 & -4 & 6 & -4 & 1
\end{pmatrix}.
\end{equation*}

The left eigenvector corresponding to the eigenvalue 1 is proportional to the stationary distribution of the balanced carries chain. This has an interpretation in terms of descents of signed permutations, analogous to Holte's interpretation of the stationary of the usual carries chain in terms of descents in ordinary permutations. To describe this, we use the linear ordering
\[ 1 < 2 < \cdots < n < -n < \cdots < -2 < -1.\] We say that
\begin{enumerate}
\item $\sigma$ has a descent at position $i$ ($1 \leq i \leq n-1$) if $\sigma(i)>\sigma(i+1)$.
\item $\sigma$ has a descent at position $n$ if $\sigma(n)<0$.
\end{enumerate}  For example, $-1 \ -2 \ -3 $ has three descents.

Let $\overline{A(n,k)}$ denote the number of signed permutations on $n$ symbols with $k$ descents. From Corollary 4.6 of \cite{DF}, one has that \[ \overline{A(n,k)} = \sum_{r=0}^k (-1)^r {n+1 \choose r} (2k-2r+1)^n .\] Hence Theorem \ref{left} implies that $v_0^n[i]$ is equal to the number of signed permutations with $i+n/2$ descents. For example when $n=4$, the entries of the left eigenvector $1 \ 76 \ 230 \ 76 \ 1$ are the number of signed permutations on $4$ symbols with $0,1,2,3,4$ descents respectively.

Next we proceed to the proof of Theorem \ref{left}.

\begin{proof} First note that $\sum_{r=0}^{i+n/2} (-1)^r {n+1 \choose r} (n+2i-2r+1)^{n-j}$ is the coefficient of $x^{2i+n+1}$ in
\[ \sum_{r \geq 0} (-1)^r {n+1 \choose r} x^{2r} \cdot \sum_{k \geq 0} k^{n-j} x^k = (1-x^2)^{n+1} \sum_{k \geq 0} k^{n-j} x^k.\]  Using the well-known fact (easily proved by induction) that \begin{equation} \label{1} \sum_{k \geq 0} k^n x^k = \left( x \frac{d}{dx} \right)^n (1-x)^{-1}, \end{equation} it follows that $v_j^n[i]$ is the coefficient of $x^{2i+n+1}$ in
\[ (1-x^2)^{n+1} \left( x \frac{d}{dx} \right)^{n-j} (1-x)^{-1}.\] Note that \[ (1-x^2)^{n+1} \left( x \frac{d}{dx} \right)^{n-j} (1-x)^{-1} \]
is a polynomial of degree $2n+1$ (being equal to $(1+x)^{n+1} (1-x)^j$ multiplied by the Eulerian polynomial $(1-x)^{n-j+1} \left( x \frac{d}{dx} \right)^{n-j} (1-x)^{-1}$ of degree $n-j$).

Thus, using the notation for $[x^a] f(x)$ as the coefficient of $x^a$ in a power series $f(x)$, it follows that
\begin{eqnarray*}
& & \sum_{i=-n/2}^{n/2} K(i,k) \cdot v^n_j[i] \\
& = & \sum_{i=-\infty}^{\infty} K(i,k) \cdot v^n_j[i] \\
& = & \frac{1}{b^n} \sum_{i=-\infty}^{\infty} \sum_{l \geq 0} (-1)^l {n+1 \choose l} [x^{kb+(n+1)(b-1)/2-i-lb}] (1-x)^{-(n+1)} \cdot v_j^n[i] \\
& = & \frac{1}{b^n} \sum_{i=-\infty}^{\infty} [x^{2i+n+1}] (1-x^2)^{n+1} (x \frac{d}{dx})^{n-j} (1-x)^{-1} \\
& & \cdot \sum_{l \geq 0} (-1)^l {n+1 \choose l} [x^{kb+(n+1)(b-1)/2-i-lb}] (1-x)^{-(n+1)} \\
& = & \frac{1}{b^n} \sum_{i=-\infty}^{\infty} [x^{2i+n+1}] (1-x^2)^{n+1} (x \frac{d}{dx})^{n-j} (1-x)^{-1} \\
& & \cdot \sum_{l \geq 0} (-1)^l {n+1 \choose l} [x^{2kb+(n+1)(b-1)-2i-2lb}] (1-x^2)^{-(n+1)} \\
& = & \frac{1}{b^n} \sum_{l \geq 0} (-1)^l {n+1 \choose l} [x^{2kb+(n+1)b-2lb}] (x \frac{d}{dx})^{n-j} (1-x)^{-1} \\
& = & \frac{1}{b^n} \sum_{l=0}^{k+n/2} (-1)^l {n+1 \choose l} (2kb+(n+1)b-2lb)^{n-j} \\
& = & \frac{1}{b^j} \sum_{l=0}^{k+n/2} (-1)^l {n+1 \choose l} (2k+(n+1)-2l)^{n-j} \\
& = & \frac{1}{b^j} v_j^n[k], \end{eqnarray*} as needed. Note that the sixth equality used \eqref{1}.
\end{proof}

Proposition \ref{recur} gives a recursive formula for the left eigenvectors of the balanced carries chain, which is very similar to that of the type $A$ Foulkes characters on page 306 of \cite{K}. For $0 \leq i,j \leq n$, define
\[ w^n_j[i] = \sum_{r=0}^i (-1)^r {n+1 \choose r} (2i-2r+1)^{n-j} .\] Note that for $n$ even, $w^n_j[i]=v^n_j[i-n/2]$.

\begin{prop} \label{recur} With notation as above,
\[ w^n_j[i] = w^{n-1}_{j-1}[i] - w^{n-1}_{j-1}[i-1] \] for all $1 \leq i,j \leq n$. Moreover there are the boundary conditions $w^n_0[i]= \overline{A(n,i)}$ and $w^n_j[n]=(-1)^j$.
\end{prop}

\begin{proof} The recurrence $w^n_j[i] = w^{n-1}_{j-1}[i] - w^{n-1}_{j-1}[i-1]$ follows from the fact in the proof of Theorem \ref{left} that $w_j^n[i]$ is the coefficient of $x^{2i+1}$ in \[ (1-x^2)^{n+1} \left( x \frac{d}{dx} \right)^{n-j} (1-x)^{-1}.\]
The equation $w^n_0[i]=\overline{A(n,i)}$ is clear from the formula for $A(n,i)$ just preceding the proof of Theorem \ref{left}. To see that $w^n_j[n]=(-1)^j$, note from the proof of Theorem \ref{left} that $w^n_j[n]$ is the coefficient of $x^{2n+1}$ in the product of $(1+x)^{n+1} (1-x)^j$ with the $n-j$th Eulerian polynomial. Since the $n-j$th Eulerian polynomial is monic of degree $n-j$, it follows that $w_j^n[n] = (-1)^j$.
\end{proof}

{\it Remark:} Comparing Theorem 5 of Miller's paper \cite{Mi} (in the case r=2 of the hyperoctahedral group) with Theorem \ref{left} shows that our left eigenvectors for the balanced carries chain are indeed equal to Miller's hyperoctahedral Foulkes characters. The recurrence in our Proposition \ref{recur} is equivalent to the recurrence in his Theorem 7, though the proof is completely different.

\[ \]

There is another derivation of the stationary distribution (left eigenvector corresponding to eigenvalue $1$) of the balanced carries chain, when an even number of numbers is added.

\begin{theorem} \label{stat} For $0 \leq j \leq n$, with $n$ fixed, \[ lim_{r \rightarrow \infty} K^r(0,j-n/2) = \frac{ \overline{A(n,j)} }{2^n n!} .\] \end{theorem}

To begin, recall the following lemma from \cite{DF}.

\begin{lemma} \label{lem} Let $U_1,\cdots,U_n$ be independent, identically distributed continuous uniform random variables on $[0,1]$. Then \[ P \left( j - \frac{1}{2} \leq U_1 + \cdots + U_n \leq j + \frac{1}{2} \right) = \frac{ \overline{A(n,j)} }{2^n n!} .\]
\end{lemma}

\begin{proof} (Of Theorem \ref{stat}) From the proof of Theorem \ref{trans}, $K^r(0,j-n/2)$ is equal to the probability that
\begin{eqnarray*} (j-n/2)b^r + \frac{(n-1)(b^r-1)}{2}
& \leq & X_1+\cdots+X_n \\
& \leq & (j-n/2)b^r + \frac{(n+1)(b^r-1)}{2},
\end{eqnarray*} where $X_1,\cdots,X_n$ are discrete uniforms on $\{ 0, \cdots, b^r-1 \}$. This is equal to the probability that
\begin{eqnarray*}
(j-n/2)b^r + \frac{(n-1)(b^r-1)}{2} & \leq & \sum_{i=1}^n Y_i - \sum_{i=1}^n (Y_i - \lfloor Y_i \rfloor) \\
& \leq & (j-n/2)b^r + \frac{(n+1)(b^r-1)}{2},
\end{eqnarray*} where $Y_1,\cdots, Y_n$ are continuous iid uniforms on $[0,b^r]$. Letting $U_i=Y_i/b^r$ be iid uniforms on $[0,1]$, it follows that $K^r(0,j-n/2)$ is equal to the probability that
\[ j-1/2-(n-1)/(2b^r) \leq \sum_{i=1}^n U_i - E \leq j+1/2-(n+1)/(2b^r), \] where $E= \sum_{i=1}^n (Y_i - \lfloor Y_i \rfloor) /b^r$. Since $E$, $(n-1)/(2b^r)$, and $(n+1)/(2b^r)$ all tend to 0 with probability 1 as $r \rightarrow \infty$ and $n,b$ are fixed, it follows from Slutsky's theorem that
\[ lim_{r \rightarrow \infty} K^r(0, j-n/2) = P \left( j-1/2 \leq \sum_{i=1}^n U_i \leq j+1/2 \right). \] Applying Lemma \ref{lem} finishes the proof of the theorem. \end{proof}

{\it Remark:} There is a representation of the transition probabilities of the balanced carries chain which might be useful for bounding the total variation convergence rate of the balanced carries chain to its stationary distribution $\pi$ (arguing as in Theorem 3.4 of \cite{DF}). Indeed, Theorem \ref{stat} and Lemma \ref{lem} give that for $0 \leq j \leq n$,
\[ \pi(j-n/2) = P \left( \lfloor U_1 + \cdots + U_n +1/2 \rfloor = j \right) .\] We want to quantify the convergence of $K^r(0,j-n/2)$ to $\pi(j)$ as $r \rightarrow \infty$. By the proof of Theorem \ref{trans}, it follows that
$K^r(0,j-n/2)$ is equal to \[ P \left( j - (n-1)/(2b^r) \leq \frac{X_1}{b^r} + \cdots + \frac{X_n}{b^r} + 1/2 \leq j +1  - (n+1)/(2b^r) \right), \] where $X_1,\cdots,X_n$ are discrete uniforms on $\{0,1,\cdots,b^r-1 \}$. For $r$ large enough with respect to $n$, this ${\it almost}$ says that \[ K^r(0,j-n/2) = P \left( \lfloor \frac{1}{b^r} \sum_{i=1}^n X_i + 1/2 \rfloor =j \right).\]

\section{Right eigenvectors of the amazing matrix for balanced carries} \label{righteig}

This section describes the right eigenvectors of the balanced carries matrix $K$. As usual, we assume that the base $b$ is odd, and that an even number of numbers are being added.

\begin{theorem} \label{right} Let $K$ be the transition matrix of the balanced carries chain of Section \ref{amazing}. The $j$th right eigenvector of $K$ (where $0 \leq j \leq n$) corresponding to the eigenvalue $1/b^j$, evaluated at the state $i$ (where $-n/2 \leq i \leq n/2$), is given by
\[ u_j^n[i] = [x^{n-j}] (x-n-2i+1)(x-n-2i+3) \cdots (x-n-2i+2n-1), \] where $[x^a] f(x)$ denotes the coefficient of $x^a$ in $f(x)$.
\end{theorem}

For example, when $n=2$, the matrix whose columns are the right eigenvectors of $K$ with eigenvalue $1/b^j$ (with $0 \leq j \leq 2$)
is given by

\begin{equation*}
\begin{pmatrix}
1 & 4 & 3 \\
1 & 0 & -1 \\
1 & -4 & 3
\end{pmatrix}.
\end{equation*}

When $n=4$, the matrix of right eigenvectors is

\begin{equation*}
\begin{pmatrix}
1 & 16 & 86 & 176 & 105 \\
1 & 8 & 14 & -8 & -15 \\
1 & 0 & -10 & 0 & 9 \\
1 & -8 & 14 & 8 & -15 \\
1 & -16 & 86 & -176 & 105
\end{pmatrix}.
\end{equation*}

{\it Remark:} The right eigenvectors allow describing the distribution of functionals of the carries chain $\{ \kappa_j \}$. For example,
the second eigenvector (corresponding to eigenvalue $1/b$) forming the second column of the matrices above, is $-2ni$. By scaling,
the function $i$ is an eigenvector. Translating into probability language, for $s$ less than $t$,
\[ E(\kappa_t|\kappa_s=m) = m/b^{t-s}.\] Similarly, the formulae of Theorem \ref{right} show that explicit polynomials of degree
$a$ in $\kappa$ are eigenvectors of the chain with eigenvalues $1/b^a$.

\begin{proof} Let $V$ denote the matrix of left eigenvectors of the balanced carries chain. Thus the $jth$ row of $V$ has $i$th entry
$v^n_j[i]$. Let $U$ denote the matrix whose columns are the right eigenvectors of the balanced carries chain; thus the $j$th column of $U$ has $i$th entry $u^n_j[i]$. We will prove that $U= 2^n n! \cdot V^{-1}$, which implies the theorem since by Theorem \ref{left}, the eigenvalues of the carries matrix are distinct.

Letting $u_{ik} = u_k^n[i]$ be the $(i,k)$ entry of $U$ (where $-n/2 \leq i \leq n/2$ and $0 \leq k \leq n$) and $v_{kj} = v_k^n[j]$ be the $(k,j)$ entry of $V$ (where $-n/2 \leq j \leq n/2$ and $0 \leq k \leq n$), one computes that
\begin{eqnarray*}
& & \sum_{k=0}^n u_{ik} v_{kj} \\
& = & \sum_{k=0}^n [x^{n-k}] (x-n-2i+1)(x-n-2i+3) \cdots (x-n-2i+2n-1) \\
& & \cdot v_{kj} \\
& = & \sum_{k=0}^n [x^{n-k}] (x-n-2i+1)(x-n-2i+3) \cdots (x-n-2i+2n-1)\\
& & \cdot \sum_{r=0}^{j+n/2} (-1)^r {n+1 \choose r} (n+2j-2r+1)^{n-k} \\
& = & \sum_{r=0}^{j+n/2} (-1)^r {n+1 \choose r} \sum_{k=0}^n (n+2j-2r+1)^{n-k} \\
& & \cdot [x^{n-k}] (x-n-2i+1)(x-n-2i+3) \cdots (x-n-2i+2n-1) \\
& = & \sum_{r=0}^{j+n/2} (-1)^r {n+1 \choose r} (2(j-i)-2r+2) \cdots (2(j-i)-2r+2n) \\
& = & 2^n \sum_{r=0}^{j+n/2} (-1)^r {n+1 \choose r} (n+(j-i)-r) \cdots (1+(j-i)-r) \\
& = & 2^n n! \sum_{r=0}^{j+n/2} (-1)^r {n+1 \choose r}{n+j-i-r \choose n} \\
& = & 2^n n! \sum_{r \geq 0} (-1)^r {n+1 \choose r}{n+j-i-r \choose n} \\
& = & 2^n n! \cdot \delta_{i,j},
\end{eqnarray*} where the final equality is explained on page 147 of \cite{Ho}. \end{proof}

{\it Remark}: The right eigenvectors are related to the type $B$ riffle shuffles studied in \cite{BB}. More precisely,
for $1 \leq k \leq n$, one has the following generating function:
\begin{eqnarray*}
& & \sum_{k=1}^n E_{n,k} x^k \\
& = & \frac{1}{2^n n!} \sum_{\pi \in B_n} (x-2d(\pi)+1)(x-2d(\pi)+3) \cdots (x-2d(\pi)+2n-1) \pi,
\end{eqnarray*}
where the $E_{n,k}$ are the Eulerian idempotents of the hyperoctahedral group $B_n$. Here $d(\pi)$ is what
\cite{BB} calls the number of descents of $\pi$ (the definition of descents in \cite{BB} is slightly different than the definition earlier
in this section). Thus
\begin{eqnarray*}
\sum_{k=1}^n E_{n,k} (x-n)^k & = & \frac{1}{2^n n!} \sum_{\pi \in B_n} (x-n-2d(\pi)+1)(x-n-2d(\pi)+3) \\
& & \cdots (x-n-2d(\pi)+2n-1) \pi.
\end{eqnarray*} Letting $E_{n,k}[i]$ denote the value of $E_{n,k}$ on a permutation with $i$ descents, it follows that
$u^n_j[i]$ is equal to $2^n n!$ multiplied by the coefficient of $x^{n-j}$ in $\sum_{k=1}^n E_{n,k}[i] (x-n)^k$.

\section{Balanced carries as a point process} \label{point}

 This section works with an odd base $b$ and digits $0,\pm 1,\pm 2,\dots,\pm (b-1)/2$. Then successive carries when adding a column of digits are $0,\pm b$. As such, independent uniformly chosen digits generate a stationary, one-dependent marked point process. Call this $X_1,X_2,\dots$. See \cite{borodin} for background on one-dependent processes.
\begin{example}
Working mod $5$, consider Table \ref{tab5}. The carries are shown in the central column as $X_1=-5$, $X_2=X_3=0$, $X_4=5,\dots$ (the carries are $\pm b$ in general). On the right are remainders $-2,1,2,\dots$. If the digits are independent and identically distributed in $\{0,\pm 1,\pm 2\}$, so are the remainders. If the remainders are $R_i$, there is a carry of $-5$ iff $R_i-R_{i+1}\leq -3$, a carry of $5$ iff $R_i-R_{i+1}\geq 3$, and a zero carry otherwise. Replace ``$3$'' by $(b+1)/2$ for general bases.
\begin{table}[htb]
\caption{Carries down a column for $b=5$ with signed digits. The right column shows the remainders. There is a $-$ or $+$ in the central column for a carry of $-b$ or $b$.}
\begin{small}\begin{equation*}\begin{array}{ccc}
\bar2&-&\bar2\\
\bar2&&1\\
1&&2\\
0&+&2\\
2&&\bar1\\
2&&1\\
\bar2&&\bar1\\
1&&0\\
1&&1\\
1&+&2\\
1&&\bar2\\\cline{1-1}
\multicolumn{2}{c}{3=1\bar2}&
\end{array}\end{equation*}\end{small}
\label{tab5}
\end{table}
\end{example}

From this description, it is easy to see that for any base, two successive $++$ or $--$ carries are impossible. For $b\geq 5$, all other patterns occur with positive probability. For $b=3$, $++$, $--$, $+0+$, and $-0-$ are impossible. The probability distribution of this balanced carries process can be expressed via determinantal formulae from \cite{borodin}. These authors determine the joint distribution for the process that records $1$ or $0$ as there is a carry or not (e.g., replace all $\pm b$ symbols by $1$). Let $a_i=P(i-1\text{ consecutive ones})$ with $a_1=1$. Theorem 4.1 in \cite{borodin} gives
\begin{theorem}
For a stationary binary one-dependent process with $a_i=P(X_1=X_2=\dots=X_{i-1}=1)$, $a_1=1$ and a binary string $t_1,t_2,\dots,t_{n-1}$ with $k$ zeros at positions $S=\{s_1<s_2<\dots<s_k\}\subseteq [n-1]$,
\begin{equation}\label{fir}
P(t_1,\dots,t_{n-1})=\det(a_{s_{j+1}-s_i})_{i,j=0}^k.
\end{equation}
The determinant is of a $(k+1)\times (k+1)$ matrix and $s_0=0$, $s_{k+1}=n$, $a_0=1$, and $a_i=0$ for $i < 0$.
\end{theorem}
\begin{example}
\[ P(0,0,0)= \det\begin{pmatrix} 1 & a_2 & a_3 & a_4 \\ 1 & 1 & a_2 & a_3 \\ 0 & 1 & 1 & a_2 \\ 0 & 0 & 1 & 1 \end{pmatrix} = 1-3a_2+a_2^2+2a_3-a_4. \]
\end{example}

Theorem 4.2 and Corollary 4.3 of \cite{borodin} give expressions for \eqref{fir} as skew Schur functions and for the higher order correlations in terms of the $a_i$.

It thus remains to determine $a_i$. These can be determined by the transfer matrix method and a serendipitous diagonalization.

\begin{prop}\label{prop:ai}
For odd $b \geq 3$ and the balanced coset representatives $0,\pm 1,\dots,\pm (b-1)/2$, $a_1=1$ and the chance $a_i$ of $i-1$ consecutive carries is

\[ a_i=\begin{cases} \frac{8}{b^{i+1}} \sum_{r=1}^{(b-1)/2} \lambda_r^{(i-1)/2} v_r^2  & \text{if }i >1 \text{ is odd}, \\
    \frac{8}{b^{i+1}} \sum_{r=1}^{(b-1)/2} \lambda_r^{(i-2)/2} v_r w_r       & \text{if }i >0 \text{ is even} \end{cases} \] where
\[ 1/\lambda_r = 4 sin^2((2r-1) \pi/2b), \ \ 1 \leq r \leq (b-1)/2, \]
\[ v_r = \sum_{j=1}^{(b-1)/2} sin((2r-1)j \pi/b), \ \ 1 \leq r \leq (b-1)/2, \] and
\[ w_r = \sum_{j=1}^{(b-1)/2} j \cdot sin((2r-1)j \pi/b), \ \ 1 \leq r \leq (b-1)/2. \]
\end{prop}
\begin{proof}
The chance of any digit sequence of length $i$ in the remainder column is $1/b^i$. For the pattern $+-+-\cdots$ of length $i-1$, sequences of digit choices $x_1,x_2,\dots,x_i$ must be chosen so that $x_1,x_2$ yield a $+$ (so $x_1-x_2\geq (b+1)/2$), $x_2,x_3$ result in a $-$ (so $x_2-x_3\leq -(b+1)/2$), and so on. Admissible sequences can be enumerated as paths in a graph. As an example, for $b=7$, the digits are $0,\pm 1,\pm 2,\pm 3$. The corresponding graph has adjacency matrix
\[ M=\begin{pmatrix} 0 & 0 & 0 & 1 & 1 & 1 \\ 0 & 0 & 0 & 0 & 1 & 1 \\ 0 & 0 & 0 & 0 & 0 & 1 \\ 1 & 0 & 0 & 0 & 0 & 0 \\ 1 & 1 & 0 & 0 & 0 & 0 \\ 1 & 1 & 1 & 0 & 0 & 0 \end{pmatrix}=\begin{pmatrix} 0 & A \\ A^T & 0 \end{pmatrix},\quad A=\begin{pmatrix} 1 & 1 & 1 \\ 0 & 1 & 1 \\ 0 & 0 & 1 \end{pmatrix}. \]
Here the rows of $M$ are labeled by the vertices $-3,-2,-1,1,2,3$ from top to botton, and the columns of $M$ are labeled by the vertices $-3,-2,-1,1,2,3$ from left to right. Thus
\[ M^2=\begin{pmatrix} AA^T & 0 \\ 0 & A^TA \end{pmatrix},\quad M^3=\begin{pmatrix} 0 & AA^TA \\ A^TAA^T & 0 \end{pmatrix},\quad \] \[ M^4=\begin{pmatrix} AA^TAA^T & 0 \\ 0 & A^TAA^TA \end{pmatrix}. \]

Suppose next that $i$ is odd. The sum of the entries in $(A^TA)^{(i-1)/2}$ counts paths resulting in $+-+-\cdots$ ($i-1$ terms). The matrix $A^TA$ has $(a,b)$ entry $\min(a,b)$. This is the correlation matrix for random walk $S_1,S_2,\dots,S_{(b-1)/2}$ with $S_i=Y_1+\dots+Y_i$ and $Y_i$ are independent with mean $0$, variance $1$. Its spectral decomposition is known \cite{AL} : the $(b-1)/2\times (b-1)/2$ matrix $A^TA$ with $(i,j)$ entry $\min(i,j)$ has
\begin{enumerate}
\item eigenvalues $\lambda_r$ with $1/\lambda_r=4\sin^2((2r-1)\pi/2b)$, $1\leq r\leq (b-1)/2$;
\item corresponding eigenvectors
\[ \psi_r(k)=\sin((2r-1)k \pi/b), 1 \leq r \leq (b-1)/2; \]
\item the eigenvectors are orthogonal with \[ \langle\psi_i,\psi_j\rangle=\sum_k\psi_i(k)\psi_j(k)=\delta_{ij}b/4.\]
\end{enumerate}

From this, for any $\ell$,
\[ (A^TA)^{\ell}_{i,j}= \frac{4}{b} \sum_{r=1}^{(b-1)/2} \lambda_r^{\ell} \psi_r(i) \psi_r(j). \]
Summing in $i$ and $j$ gives
\[ \mathbf{1}^T(A^TA)^{\ell}\mathbf{1}= \frac{4}{b} \sum_{r=1}^{(b-1)/2} \lambda_r^{\ell} v_r^2,\quad v_r=\sum_j \psi_r(j) = \sum_{j=1}^{(b-1)/2} sin((2r-1)j \pi/b). \]
This gives the result for an even number of steps (so for $i$ odd), since the probability of any length $i$ digit sequence in the
remainder column is $1/b^i$, and one must multiply by 2 to account for the string $+-+- \cdots +-$ ($i-1$ terms) and $-+-+ \cdots -+$ ($i-1$ terms).

Suppose next that $i$ is even. By arguing as in the $i$ odd case, to compute the chance of $-+ \cdots -$ ($i-1$ terms), we need to study
\[ \mathbf{1}^T A(A^TA)^{\ell}\mathbf{1} .\] Since $\mathbf{1}^T A = (1, 2, \cdots ,(b-1)/2)$ we get that
\begin{eqnarray*}
\mathbf{1}^T A(A^TA)^{\ell}\mathbf{1} & = & \frac{4}{b} \sum_{r=1}^{(b-1)/2} \lambda_r^{\ell} \sum_{i=1}^{(b-1)/2} \sum_{j=1}^{(b-1)/2} i \psi_r(i) \psi_r(j) \\
& = & \frac{4}{b} \sum_{r=1}^{(b-1)/2} \lambda_r^{\ell} v_r w_r,
\end{eqnarray*} where
\[  v_r = \sum_{j=1}^{(b-1)/2} sin((2r-1)j \pi/b) \ , \ w_r = \sum_{j=1}^{(b-1)/2} j \cdot sin((2r-1)j \pi/b).                          \] To obtain the theorem, we now set $\ell = (i-2)/2$, divide by $b^i$ (the probability of any length $i$ digit sequence in the remainder column), and multiply by $2$ (to also account for the length $i-1$ sequence $+- \cdots +$).
\end{proof}
\begin{example}[Balanced ternary]
With $b=3$ and coset representatives $0,\pm 1$, the matrix $M$ is $\left(\begin{smallmatrix} 0 & 1 \\ 1 & 0 \end{smallmatrix}\right)$. Thus $M^2=\text{Id}$, $M^3=M$, and so on. It follows directly that the chance of $+-+-\cdots$ with length $i-1$ is $a_i=1/3^i$ for all $i$. For the zero/one consolidation the chance of $i-1$ ones is $2/3^i$.
\end{example}
\begin{example}[Base $5$]
With $b=5$ and coset representatives $0,\pm 1,\pm 2$, the matrix $M$ is
\[ M=\begin{pmatrix} 0 & A \\ A^T & 0 \end{pmatrix},\quad A=\begin{pmatrix} 1 & 1 \\ 0 & 1 \end{pmatrix} \]
with
\[ A^TA=\begin{pmatrix} 1 & 1 \\ 1 & 2 \end{pmatrix},\quad (A^TA)^2=\begin{pmatrix} 2 & 3 \\ 3 & 5 \end{pmatrix},\quad (A^TA)^3=\begin{pmatrix} 5 & 8 \\ 8 & 13 \end{pmatrix},\quad\cdots \]
We recognize $(A^TA)^h=\left(\begin{smallmatrix} F_{2h-1} & F_{2h} \\ F_{2h} & F_{2h+1} \end{smallmatrix}\right)$, where $F_i$ is the $ith$ Fibonacci number. The sum of these entries is $F_{2h-1}+2F_{2h}+F_{2h+1}=F_{2h+3}$. So $a_{2i+1}= 2 F_{2i+3}/5^{2i+1}$. Similarly $a_{2i}= 2 F_{2i+2}/5^{2i}$. This gives $a_i=2 F_{i+2}/5^i$ for all $i$:
\[ a_i=\frac{2 F_{i+2}}{5^i}=10 \sqrt{5}\left[\left(\frac{1+\sqrt{5}}{10}\right)^{i+2}-\left(\frac{1-\sqrt{5}}{10}\right)^{i+2}\right]. \]
\end{example}

\section{Acknowledgements} Diaconis was supported by NSF grant DMS 08-04324. Fulman was supported by NSA grant H98230-13-1-0219. The authors thank Danny Goldstein, Robert Guralnick, and Eric Rains for encouraging them to study balanced carries, and Xuancheng Shao and Kannan Soundararajan for their help with Section \ref{point}.


\begin{thebibliography}{AAA}

\bibitem{AL} Akesson, F. and Lehoczky, J., Discrete eigenfunction expansion of multi-dimensional Brownian motion and the Ornstein-Uhlenbeck
process, (1998), unpublished manuscript.

\bibitem{A} Alon, N., Minimizing the number of carries in addition, {\it SIAM J. Discrete Math.} {\bf 27} (2013), 562-566.

\bibitem{BB} Bergeron, F. and Bergeron, N., Orthogonal idempotents in the descent algebra of $B_n$ and applications, {\it J. Pure Appl. Algebra} {\bf 79} (1992), 109-129.


\bibitem{borodin} Borodin, A., Diaconis, P., and Fulman, J., On adding a list of numbers (and other one-dependent determinantal processes), {\it Bull. Amer. Math. Soc.} {\bf 47} (2010), 639-670.

\bibitem{BW} Brenti, F. and Welker, V., The Veronese construction for formal power series and graded algebras, {\it Adv. in Appl. Math.} {\bf 42} (2009), 545-556.

\bibitem{Ca} Cajori, F., {\it A history of mathematical notations}, Volume 1. Dover Publications, 1993.

\bibitem{Co} Colson, J., A short account of negative-affirmativo arithmetik, {\it Philos. Transac. of the Royal Society} {\bf 34} (1726), 161-173.

\bibitem{DF} Diaconis, P. and Fulman, J., Carries, shuffling, and symmetric functions, {\it Adv. Appl. Math} {\bf 43} (2009), 176-196.

\bibitem{DF2} Diaconis, P. and Fulman, J., Carries, shuffling, and an amazing matrix, {\it Amer. Math. Monthly} {\bf 116} (2009), 788-803.

\bibitem{DF3} Diaconis, P. and Fulman, J., Foulkes characters, Eulerian idempotents, and an amazing matrix, {\it J. Algeb. Combin.} {\bf 36} (2012),
425-440.

\bibitem{DPR} Diaconis, P., Pang, A., and Ram, A., Hopf algebras and Markov chains: two examples and a theory, to appear in {\it J. Algebr. Comb.}

\bibitem{DSS} Diaconis, P., Shao, X., and Soundararajan, K., Carries, group theory, and additive combinatorics, arXiv:1309.0434 (2013).

\bibitem{Ho} Holte, J., Carries, combinatorics, and an amazing matrix, {\it Amer. Math. Monthly} {\bf 104} (1997), 138-149.

\bibitem{I} Isaksen, D., A cohomological viewpoint on elementary school arithmetic, {\it Amer. Math. Monthly} {\bf 109} (2002), 796-805.

\bibitem{K} Kerber, A., Applied finite group actions. Second edition. Algorithms and Combinatorics, 19. Springer-Verlag, Berlin, 1999.

\bibitem{Kn} Knuth, D., The art of computer programming. Vol. 2. Seminumerical algorithms. Third edition. Addison-Wesley, Reading Mass, 1998.

\bibitem{Mi} Miller, A., Foulkes characters for complex reflection groups, preprint (2013).

\bibitem{NS} Nakano, F. and Sadahiro, T., A generalization of carries process and Eulerian numbers, arXiv:1306.2790 (2013).

\bibitem{NT} Novelli, J-C and Thibon, J-Y, Noncommutative symmetric functions and an amazing matrix, {\it Adv. Appl. Math.} {\bf 48}
(2012), 528-534.

\bibitem{P} Pang, A., A Hopf-power Markov chain on compositions, in {\it 25th International Conference on Formal Power Series
and Algebraic Combinatorics}, Discrete Math. Theor. Comput. Sci. Proc., (2013), 499-510.

\end{thebibliography}
\end{document}